\documentclass[12pt,twoside]{amsart}

\usepackage[usenames]{color}

\usepackage{latexsym}
\usepackage{amssymb}
\usepackage{amsfonts}
\usepackage{amstext}
\usepackage{multicol}
\usepackage{amsmath, amsthm}
\usepackage{setspace}
\usepackage{appendix}

\newtheorem{theorem}{Theorem}[section]
\newtheorem{lemma}[theorem]{Lemma}
\newtheorem{corollary}[theorem]{Corollary}
\newtheorem{proposition}[theorem]{Proposition}
\newtheorem{definition}[theorem]{Definition}
\newtheorem{remark}[theorem]{Remark}
\newtheorem{assumption}[theorem]{Assumption}

\def\P{{\mathbb P}}
\def\E{{\mathbb E}}
\def\V{{\Vert}}
\def\v{{\vert}}
\def\R{{\mathbb R}}
\def\HS{{\rm HS}}
\def\L{{\mathcal L}}
\def\N{{\mathcal N}}
\def\vv{\varepsilon}
\def\span{{\rm span}}

\begin{document}
\title{Upper bound for intermediate singular values of random matrices}
\date{}
\author{Feng Wei}
\thanks{Partially supported by M. Rudelson's  NSF Grant DMS-1464514,  and USAF Grant FA9550-14-1-0009. }

\maketitle

\begin{abstract}
In this paper, we prove that an $n\times n$
matrix $A$ with independent centered subgaussian entries satisfies
\[
  s_{n+1-l}(A) \le   C_1t \frac{l}{\sqrt{n}}
\]
 with probability at least $1-\exp(-C_2tl)$. This yields $s_{n+1-l}(A) \sim  \frac{cl}{\sqrt{n}}$ in combination with a known lower bound. These results can be generalized to the rectangular matrix case.

\end{abstract}

\section{Introduction}

Consider an $n\times m$ real matrix $A$ with $n\ge m$. The singular values $s_k(A)$ of $A$, where $k=1,2,\cdots, n$, are the eigenvalues of $\sqrt{A^T A}$ arranged in non-increasing order.
The non-asymptotic singular value distribution of random i.i.d. sub-gaussian matrix is an important and interesting subfield in random matrix theory.
The first result in this direction was obtained in \cite{InvofRM}, where it was proved that the smallest singular value of a square i.i.d. sub-gaussian matrix is bounded below by $n^{-3/2}$ with high probability.
This result was later extended and improved  in a number of papers, including \cite{TaoVuAnnals,TVsmallestSVdist,LOandInvofRM,InvofsparsenonHerm,Invofheavytail}.
The above mentioned results pertain to square matrices.
 A probabilistic lower bound for the smallest singular value  of a  rectangular matrix was obtained by M.~Rudelson and R.~Vershynin \cite{SMstSvRect}. They proved that an $n\times (n-l)$ matrix has smallest singular value lower bounded by $\frac{\vv l}{\sqrt{n}}$ with probability at least $1-\left(C\vv \right)^{l}-\exp(-Cn)$. Using this result, one can show that for a square i.i.d. sub-gaussian matrix $A$, $s_{n+1-l}(A)>c\frac{l}{\sqrt{n}}$ with high probability.

However, the optimal upper bound of the singular values for general sub-gaussian matrices is unknown. Prior to this paper, the only progress in this direction was made by Szarek \cite{DisttoLinfSZ} who proposed an optimal upper bound for the gaussian matrix. Szarek proved that for a standard gaussian i.i.d. matrix $G$, $\frac{Cl}{\sqrt{n}}\le s_{n+1-l}(G)\le \frac{Cl}{\sqrt{n}}$ with probability at least  $1-\exp(-Cl^2)$. This result suggests that $l$th smallest singular value of an i.i.d. sub-gaussian matrix is concentrated around $\frac{l}{\sqrt{n}}$.

Although the optimal upper bound is not proved for general matrices, some  results can be deduced.
T. Tao and V. Vu have established the universal behavior of small singular values in \cite{TVsmallestSVdist} (see Theorem 6.2 \cite{TVsmallestSVdist}). Combined with Szarek's Theorem 1.3 in \cite{DisttoLinfSZ}, their approach allows us to deduce some non-asymptotic bounds for random i.i.d. square matrix under a moment condition. However, their bound only works for $l\le n^c$ where $c$ is a small constant. Tao and Vu's approach  \cite{TVsmallestSVdist} is based on the Berry-Esseen Theorem for the frames and does not allow one to obtain exponential bounds for the probability as we do in our Theorem \ref{upperbound}.
Also,  C. Cacciapuoti, A. Maltsev, B. Schlein estimated the rate of convergence of the empirical measure of singular values to the limit distribution near the hard edge (see \cite{NOUBofsvs} Theorem 3). Theorem 3 in \cite{NOUBofsvs} can be used to derive an upper bound of the form $\frac{cl^C}{\sqrt{n}}$ \cite{NOUBofsvs}.  Better understood is the upper bound for the smallest singular value. M. Rudelson and R. Vershynin were the first to prove the smallest singular value of the i.i.d. sub-gaussian matrix is also bounded from above by $\frac{c}{\sqrt{n}}$ with high probability (see \cite{UBssvMR}). A different proof with an exponential tail probability can be found in a very recent paper by H. Nguyen and V. Vu \cite{SsvUB2}.

In this paper, we prove the upper bound on the singular values under two assumptions:  that the entries of the matrix are non-degenerate; and that they have a fast tail decay. The first assumption is quantified in terms of the Levy concentration function and the second is quantified  in  terms of the $\psi_\theta$-norm. Next we provide definitions.
\begin{definition}
Let $Z$ be a random vector that takes values in $\mathbb{C}^n$. The concentration function of $Z$ is defined as
$$
\mathcal{L}(Z,t)=\sup_{u\in \mathbb{C}^n} \P \{\V Z- u \V_2 \le t\}, t\ge 0.
$$
\end{definition}

\begin{definition}
Let $\theta>0$. Let $Z$ be a random variable on a probability space $(\Omega, \mathcal{A},\P)$. Then the $\psi_{\theta}$-norm of $Z$ is defined as
$$
\V Z \V_{\psi_\theta}:=\inf\left\{ \lambda>0: \E \exp\left( \frac{| Z |}{\lambda} \right)^{\theta}   \le 2\right\}
$$
\end{definition}
If $\V X \V_{\psi_\theta} < \infty$, then $X$ is called a $\psi_\theta$ random variable. This condition is satisfied for broad classes of random variables. In particular, a bounded random variable is $\psi_\theta$ for any $\theta>0$, a normal random variable is $\psi_2$, and a Poisson variable is $\psi_1$.

In this paper, we prove that for all $l$, $s_{n+1-l}(A) \ge \frac{Ctl}{\sqrt{n}}$ with an exponentially small probability, where $A$ is a random matrix under the following assumption:
\begin{assumption}\label{asump1}
Let $p>0$.
Let $A$ be an $n\times m$ random matrix with i.i.d. entries that have mean 0, variance 1 and $\psi_2$-norm $K$. Assume also  that
there exists $0<s \le s_0(p,K)$ such that
$$
  \L(A_{i,j}, s) \le p s.
$$
Here, $s_0(p,K)$ is a given function depending only on $p$ and $K$.
\end{assumption}
A concrete value of $s_0(p,K)$ can be detemined by tracing the proof of Theorem \ref{upperbound}.

\begin{remark}{\rm
The condition on the Levy concentration function is automatically satisfied if the density of the entries is bounded by $p$. However, our result holds in a much more general setting because we require this condition to hold only for one fixed $s$  and not for all $s>0$. This assumption can be viewed as a discrete analog of the bounded density condition.}
\end{remark}

\begin{remark}{\rm
The analysis of the proof for Theorem \ref{upperbound} shows that it is enough to take $s_0(p,K)=c(K)\min \{ p^{-1}, 1 \}$, where $c(K)$ is a small constant that depends only on $K$.
}
\end{remark}

We prove the following main theorem for a random matrix $A$ satisfying Assumption \ref{asump1}:
\begin{theorem}{\rm (Upper bound for singular values of an i.i.d. sub-gaussian square matrix)}\label{upperbound}
Let $A$ be an $n\times n$ random matrix that satisfies Assumption \ref{asump1} with some  $s_0(p,K)$ that depends only on $p,K$. Then there exist constants $C_1,C_2>0$ such that for all $t>1$, for all $l$ between 1 and $n$,
$$
\P\left( s_{n+1-l}(A)\le C_1 \frac{t l}{\sqrt{n}} \right) \ge 1-\exp(-C_2 t l)
$$
where $C_1,C_2$ are constants that depend only on $K,p$.
\end{theorem}

\begin{remark}{\rm In \cite{SsvUB2}, Nguyen and Vu obtained a sharp bound for the smallest singular value of i.i.d. sub-gaussian matrices with exponential tail probability. Theorem \ref{upperbound} recovers the result obtained by Nguyen and Vu under Assumption \ref{asump1} and generalizes that result to all $l$. }
\end{remark}

\begin{remark}{\rm Unlike the probability bound deduced from \cite{TVsmallestSVdist}, our probability tail bound is exponential type. Szarek's probability estimate \cite{DisttoLinfSZ} suggests that the optimal probability bound for $\P \left( s_{n+1-l}(A)\le \frac{c l}{\sqrt{n}} \right)$ is $1-\exp(-Cl^2)$. }
\end{remark}

\begin{remark}{\rm Possible applications of Theorem \ref{upperbound} include eigenvector $l_{\infty}$ delocalization of random matrices \cite{Linfdeloc}. For that, one has to consider $A-zI$ instead of the matrix $A$. Further effort would be needed to generalize our result to the case with a shift. }
\end{remark}

Also, Theorem \ref{upperbound} can be extended to rectangular matrices easily (see Section 4). And more precisely, we have the following corollary:
\begin{corollary}{\rm (Upper bound for singular values of an i.i.d. sub-gaussian matrix)}\label{upperboundRect}
Let $A$ be an $n\times (n-k) $ random matrix that satisfies Assumption \ref{asump1}  with some  $s_0(p,K)$ that depends only on $p,K$. Then there exist constants $C_1,C_2>0$ such that for all $ t>1$ and $l$ between 1 and $n$,
$$
\P\left( s_{n+1-l}(A)\le C_1\frac{t l}{\sqrt{n}} \right) \ge 1-\exp(-C_2 t l)
$$
where $C_1,C_2$ are constants that depend only on $K,p$.
\end{corollary}

A direct application of Theorem \ref{upperbound} and Theorem \ref{leastSVLB} leads us to a generalization of Theorem 1.3 in \cite{DisttoLinfSZ}.
\begin{corollary}{\rm (Non-asymptotic singular values distribution of i.i.d. sub-gaussian square matrix)}\label{mainthm}
Let $A$ be an $n\times n$ random matrix that satisfies Assumption \ref{asump1} with some  $s_0(p,K)$ that depends only on $p,K$. Then there exist $0 < C_1<C_2$ and $C_3>0$, such that for all $l$ between 1 and $n$,
$$
\P\left( \frac{C_1 l}{\sqrt{n}} \le s_{n+1-l}(A) \le \frac{C_2 l}{\sqrt{n}}  \right) \ge  1- \exp(-C_3 l)
$$
where $C_i$s are  constants that depend only on $K,p$.
\end{corollary}

A similar proof can lead to an analog in the rectangular case:
\begin{corollary}{\rm (Non-asymptotic distribution of singular values in the i.i.d. sub-gaussian rectangular matrix)}\label{mainthmrect}
Let $A$ be an $n\times (n-k)$ random matrix that satisfies Assumption \ref{asump1}  with some  $s_0(p,K)$ that depends only on $p,K$. Then there exist $ 0 < C_1<C_2$ and $C_3>0$, such that for all $l$ between  $k$ and $n$,
$$
\P\left( \frac{C_1 l}{\sqrt{n}} \le s_{n+1-l}(A) \le \frac{C_2 l}{\sqrt{n}}  \right) \ge  1- \exp(-C_3 l)
$$
where $C_i$s are  constants that depend only on $K,p$.
\end{corollary}

In Section 2, we present some preliminary results that are needed to prove Theorem \ref{upperbound}. In Section 3, we provide proof of Theorem \ref{upperbound}. We prove Theorems \ref{upperboundRect} and \ref{mainthm} In Section 4.

\section{Notation and Preliminaries}

Throughout this paper, $c$ denotes absolute constants, $C$ denotes constants that may depend only on the parameters $K,p$. Note that these constants may vary from line to line. $S^{n-1}$ denotes the $n$  dimensional sphere, i.e., the sphere in $\R^n$ which itself is an $(n-1)$-dimensional manifold. $S_E$ denotes the sphere of a subspace $E$, i.e., $S_E=S^{n-1}\cap E$.

For a $n\times n$ random matrix $A$, as in Theorem \ref{upperbound}, we denote by $A_l$ the $n\times l$  matrices of the first $l$  columns of $A$. $A_{n-l}$ denotes the $n\times (n-l)$ matrix of the last $n-l$ columns of $A$. Without loss of generality, we can assume $A$ is a$.$s. invertible. If not we prove the theorem for $A+\vv G$, where $G$ is an independent gaussian matrix. Then the result holds for $A$ up to an abssolute constant by sending $\vv$ to zero. $X_k$ will denote the $k$th column of matrix $A$, and we use the following notations
\begin{itemize}
\item $H_l:=$span$(X_j)_{j>l}$.
\item $H_{l,k}:=$span$(X_j)_{j>l,j\neq k}$.
\item $P_l,P_{l,k}$ are the orthogonal projections onto $H_l,H_{l,k}$, respectively.
\item $P_l^{\perp},P_{l,k}^{\perp}$ are the orthogonal projections onto $H_l^{\perp},H_{l,k}^{\perp}$, respectively.
\item $X_k^*:=(A^{-1})^*e_k$, i.e., the $k$-th column of $(A^{-1})^*$.
\item $Y_k^*:=P_l X_k^*$, $k=l+1,l+2,\cdots,n$.
\end{itemize}

\subsection{Biorthogonal system}
Consider vectors $(v_k)_{k=1}^n$ and $(v_k^*)_{k=1}^n$ that belong to an $n-$dimensional Hilbert space $H$. Recall that the system $(v_k,v_k^*)_{k=1}^n$ is called a biorthogonal system in $H$ if $\langle v_j,v_k^*\rangle =\delta_{j,k}$ for all $j,k=1,2,\cdots,n$. The system is called complete if span$(v_k)=H$.
The following proposition summarizes some elementary known properties of biorthogonal systems. For the reader's convenience, we provide the proof of this elementary fact in the appendix.
\begin{proposition}\label{BOSp}
\begin{enumerate}
\item  Let $D$ be an $n\times n$ invertible matrix with columns $v_k=De_k$, $k=1,2,\cdots,n$. Define $v_k^*=(D^{-1})^*e_k$. Then $(v_k,v_k^*)_{k=1}^n$ is a complete biorthogonal system in $\mathbb{R}^n$.\\
\item  Let $(v_k)_{k=1}^n$ be a linearly independent system in an $n-$dimensional Hilbert space $H$. Then there exist unique vectors $(v_k^*)_{k=1}^n$ such that $(v_k,v_k^*)_{k=1}^n$ is a biorthogonal system in $H$. This system is complete.\\
\item  Let $(v_k,v_k^*)_{k=1}^n$ be a complete biorthogonal system in a Hilbert space $H$. Then $\V v_k^* \V_2 = 1/{\rm dist}(v_k,{\rm span}(v_j)_{j\neq k}) $ for $k=1,2,\cdots,n$.
\end{enumerate}
\end{proposition}
The next Lemma tells us the relation between $Y_k^*$ and $X_k^*$ for $k\ge l+1$.
\begin{lemma}\label{newBOS}
$(X_k,Y_k^*)_{k=l+1}^n$ is a complete biorthogonal system in $H_l$.
\end{lemma}
\begin{proof}
By definition, for all $k\ge l+1$
$$
Y_k^*-X_k^*\in \ker(P_l)=H_l^{\perp}=\span(X_j^*)_{j\le l}.
$$
So we have, for all $k\ge l+1$
$$
Y_k^*= X_k^* + \sum_{j=1}^{l}a_{kj} X_j^*, {\rm \  for \  some\ } a_{jk}\in \mathbb{R}, j=1,2,\cdots,l.
$$
By the orthogonality, we have for all $k,i\ge l+1$
$$
\langle Y_k^*, X_i \rangle = \langle X_k^*, X_i \rangle + \sum_{j=1}^{l}a_{kj} \langle X_j^*, X_i \rangle =  \langle X_k^*, X_i \rangle =\delta_{k,i}.
$$
Thus the biorthogonality is proved. The competeness follows since $\dim(H_l)=n-l$.
\end{proof}
In view of the uniqueness of Part 2 of Proposition \ref{BOSp}, Lemma \ref{newBOS} has the following crucial consequence.
\begin{corollary}\label{LIofnewBOS}
The system of vectors $(Y_k^*)_{k=l+1}^n$ is uniquely determined by the system $(X_k)_{k=l+1}^n$. In particular, the random vector system $(Y_k^*)_{k=l+1}^n$ is independent with random vector system $(X_k)_{k=1}^l$.
\end{corollary}

\subsection{Concentration thereoms}

The major tools of our proof come from measure concentration theory. Here we list the concentration theorems that will be used in the proof.

The first theorem is a concentration property of sub-gaussian random vectors.
\begin{theorem}\label{subgVconc}
Let $D$ be a fixed $m\times n$ matrix. Consider a random vector $Z$ with independent entries that have mean 0, variance greater than 1, and uniformly bounded by $K$ in $\psi_2$ norm. Then, for any $t\ge0$, we have
$$
\P( |\V DZ\V_2-M|>t )\le 2\exp\left( - \frac{ct^2}{\V D\V^2}\right)
$$
where $M=(\E \V DZ \V_{2}^{2})^{1/2}$ which satisfies $\V D \V_{{\rm HS}} \le M \le K \V D\V_{{\rm HS}}$,
and $c=c(K)$ is polynomial in $K$.
\end{theorem}
This result can be deduced from the Hanson-Wright inequality.
A modern proof of the Hanson-Wright inequality and a deduction of the above Theorem \ref{subgVconc} are discussed in \cite{HansonWrightineq}.

Sub-gaussian concentration paired with a standard covering argument yields the following result on norms of random matrices, see \cite{HansonWrightineq}.
\begin{theorem}\label{subgOPconc}
{\rm (Products of random and deterministic matrices). }Let $D$ be a fixed $m\times N$ matrix, and let $G$ be an $N\times k$ random matrix with independent entries that satisfy $\E G_{ij}=0, \E G_{ij}^2 \ge 1$ and $\V G_{ij}\V_{\psi_2} \le K$. Then for any $s,t\ge 1$ we have
$$
\P\{\V DG\V> C (s\V D \V_{\HS} + t\sqrt{k} \V D\V ) \}\le 2\exp(-s^2r-t^2k)
$$
Here $r=\V D\V_{\HS}^2/ \V D\V_2^2$ is the stable rank of $D$, and $C=C(K)$ is a polynomial in $K$.
\end{theorem}

The following result gives the lower bound on the smallest singular value of a rectangular i.i.d. sub-gaussian matrix. This will be used in our proof of Theorem \ref{upperbound}; it can also yield the lower bound in Theorem \ref{mainthm} directly. The proof and extensions of the theorem are discussed in \cite{InvofRM, LOandInvofRM, SMstSvRect, NAtheoryofRM}.

\begin{theorem}\label{leastSVLB}
Let $G$ be an $N\times n$ random matrix, $N\ge n$, whose elements are independent copies of a centered sub-gaussian random variable with unit variance. Then for every $\vv >0$, we have
$$
\P \left( s_n(G) \le \vv \left( \sqrt{N}-\sqrt{n-1}\right)\right)\le (C\vv)^{N-n+1}+e^{-C'N}
$$
where $C,C'>0$ depend (polynomially) only on the sub-gaussian moment $K$.
\end{theorem}
As one step towards the above least singular value bound, the following distance to a random subspace theorem was proved by M. Rudelson and R. Vershynin \cite{SMstSvRect}:
\begin{theorem}\label{disttoSubsp}
{\rm (Distance to a random subspace).} Let $Z$ be a vector in $\R^N$ whose coordinates are independent and identically distributed centered sub-gaussian random variables with unit variance. Let $H$ be a random subspace in $\R^N$ spanned by $N-m$ vectors, $0<m< \tilde cN$, whose coordinates are independent and identically distributed centered sub-gaussian random variables with unit variance, independent of $Z$. Then, for every $v\in \R ^N$ and every $\vv >0$, we have
$$
\P ({\rm dist}(Z,H+v) < \vv \sqrt{m})\le (C\vv )^m +e^{-cN},
$$
where $C,c, \tilde c$ depend only on the sub-gaussian moments.
\end{theorem}

M. Rudelson and R. Vershynin have recently proved the following results for small ball probability of a linear image of high dimensional distributions \cite{SBprobofLIofHD} (see also \cite{ShparbdMargDens}).
\begin{theorem}\label{ConcofProj}
{\rm (Concentration function of projections.)} Consider a random vector $Z=(Z_1,\cdots,Z_n)$ where $Z_i$ are real-valued independent random variables. Let $t,p\ge 0$ be such that
$$
\L (Z_i,t)\le p \ {\rm for\ all\ } i=1,\cdots , n
$$
Let $P$ be an orthogonal projection in $\R^n$ onto a $d$-dimensional subspace. Then
$$
\L(PZ,t\sqrt{d})\le (cp)^d.
$$
where $c$ is an absolute constant.
\end{theorem}
In the same paper, Rudelson and Vershynin generalized Theorem \ref{ConcofProj} to general matrices:
\begin{theorem}{\rm (Concentration functions of anisotropic distributions.)}\label{SBofLIofHDdistr}
Consider a random vector $Z$ where $Z_i$ are real-valued independent random variables. Let $t,p\ge 0$ be such that
$$
\L (Z_i,t)\le p \ {\rm for\ all\ } i=1,\cdots , n
$$
Let $D$ be an $m\times n$ matrix and $\vv\in(0,1)$. Then
$$
\L (DZ,t\V D\V_{\HS})\le (c_{\vv}p)^{(1-\vv)r(D)}
$$
where $r(D)=\V D\V_{\HS}^2/ \V D\V_2^2$ and $c_{\vv}$ depend only on $\vv$.
\end{theorem}

As a special case of Theorem \ref{ConcofProj}, the following corollary usefully controls the concentration function of sums:

\begin{theorem}{\rm (Concentration function of sums.)}\label{ConofSum}
Consider a random vector $Z=(Z_1,\cdots,Z_n)$ where $Z_i$ are real-valued independent random variables. Let $t,p\ge 0$ be such that
$$
\L (Z_i,t)\le p \ {\rm for\ all\ } i=1,\cdots , n
$$
Let $a_1,\cdots,a_n$ be real numbers with $\sum_{j=1}^n a_j^2 =1$. Then
$$
\L\left( \sum_{i=1}^{n} a_i Z_i, t \right)\le cp.
$$
where $c$ is an absolute constant.
\end{theorem}

\section{Proof of Theorem \ref{upperbound}}

Before proving the theorem, let us explain our strategy.
We prove a lower bound for $s_l(A^{-1})$, rather than proving an upper bound for $s_{n+1-l}(A)$. To do this, we show that there exists an $l$ dimensional subspace, such that the smallest singular value of the operator restricted on this subspace is bounded from below. Our target subspace will be $H_l^{\perp}$.

The proof uses a net argument for a specially constructed net. In Step 1, we obtain a small ball probability estimate for a random vector. A generic vector in $H_l^{\perp}$ can be represented as $A^{-1} P_l^{\perp}A_ly$ for some $y \in \R^l$.
We will show that $\frac{\V A^{-1} P_l^{\perp}A_ly \V_2}{\V P_l^{\perp}A_l y\V_2}$ is bounded from below by $\frac{C\sqrt{n}}{l}$ for any $y\in S^{l-1}$. In steps 2 and 3, we provide a union bound argument.

There are three essential features of our proof. First, as $H_l^{\perp}$ is a random subspace, we cannot consider a net on $S_{H_l^{\perp}}$ directly. So, we consider a net $\mathcal{N}_{\vv}$ on $S^{l-1}$ instead, which will induce a net on $H_l^{\perp}$. Second, to complete the argument we need to show the union bound probability of the form $\v \mathcal{N}_{\vv} \v \exp(-Cl)$ is small, where $C$ is a small constant. Since $\v \mathcal{N}_{\vv} \v \sim \left( \frac{3}{\vv} \right)^l$, this
bound in general can be large. To control the probability, we work not on $y\in S^{l-1}$ but on $y\in S^{l'-1}$, $l'=\lfloor \kappa l \rfloor$
for some $\kappa \in (0,1)$ instead. With this dimension reduction argument, we end up proving that $s_{l'}(A^{-1})\ge \frac{C\sqrt{n}}{l}$, and then we rephrase it. Third, representing a vector from $H_l^{\perp}$ as $A^{-1} P_l^{\perp}A_ly$  is advantageous because
$$
\V A^{-1} P_l^{\perp}A_ly \V_2^2=\V B A_l y \V_2^2+1
$$
where $B$ is a random matrix that is independent of $A_l$. This allows us to analyze the property of $B$ first and then apply tools like Theorem \ref{SBofLIofHDdistr} and Theorem \ref{subgOPconc}. Note that this construction was generalized from the one dimensional case presented by M. Rudelson and R. Vershynin \cite{UBssvMR}.

In the proof, we will use the well-known estimate that there exists an $\vv$-net on $S^{l-1}$ with cardinality less than $\left( \frac{3}{\vv} \right)^l$, see, e.g.,  Lemma 4.3 in \cite{LecnotesRMMR}.

\begin{proof}[Proof of Theorem \ref{upperbound}]

To prove Theorem \ref{upperbound}, we only need to prove the following claim:

{\bf Claim.} There exist $C_1$ and $C_2$ that only depend on $K,p$ such that for every $l$ between 1 and $n$,

$$
\P\left( s_{n+1-l}(A) \ge C_1\frac{l}{\sqrt{n}} \right)
\le \exp(-C_2 l ).
$$
To start, we derive Theorem \ref{upperbound} from the claim.
Let $t\ge 1$, and let $k$ be any integer between 1 and $n$. Set $l=\lfloor tk \rfloor $ and assume for a moment that $l<n$. Then
\begin{equation}
\begin{array}{rl}
\displaystyle \P \left( s_{n+1-k}(A) \ge C_1 \frac{2tk}{\sqrt{n}} \right)
\displaystyle & \le \P \left( s_{n+1-k}(A) \ge \frac{C_1 l}{\sqrt{n}} \right)\\
\displaystyle & \le \exp(-C_2l) \le \exp\left(-\frac{C_2tk}{2}\right).
\end{array}
\end{equation}
In the case $l\ge n$, the sub-gaussian tail estimate for the norm of a random matrix (one may also consider this as a special case of Theorem \ref{subgOPconc}) yields
\begin{equation}
\begin{array}{rl}
\displaystyle \P \left( s_{n+1-k}(A) \ge C_3 \frac{2tk}{\sqrt{n}} \right)
\displaystyle & \le \P \left( s_{1}(A) \ge C_3 \frac{2tk}{\sqrt{n}} \right)\\
\displaystyle & \le \exp\left(-C \frac{C_3^2t^2k^2}{n} \right)
\le \exp\left(-C_4 tk\right),
\end{array}
\end{equation}
and therefore for all $k$ between 1 and $n$,
$$
\P \left( s_{n+1-k}(A) \ge C_5 \frac{tk}{\sqrt{n}} \right)
\le \exp(-C_6tk)
$$
with constants $C_5,C_6$ depending on $p,K$ only.
So Theorem \ref{upperbound} is implied by the claim.
\end{proof}

Now, we prove the above claim.
\begin{proof}[Proof of the claim.]
In the proof of the claim, we first assume $l\le \frac{\tilde{c}n}{2}$, where $\tilde{c}$ is the same as the $\tilde{c}$ which appeared in Theorem \ref{disttoSubsp}. If $l>Cn$, then the required bound follows from the estimate for $s_1(A)$. This is a standard estimate of the operator norm that can be found in many places, for example, in Theorem 2.4 of \cite{SMstSvRect}.
Let $\alpha>1, \delta, \kappa <1,\beta<\alpha^{-1}<1$ be parameters to be chosen later.
Also,  assume that
\begin{equation} \label{concCondition}
  \L(A_{i,j}, \beta ) \le p \beta
\end{equation}
i.e. Assumption \ref{asump1} is true with $s=s_0(p,K)=\beta$.

{\bf Step 1.} {\bf Concentration for a random vector}.
Consider $y \in S^{l-1}$, define
\begin{equation}  \label{eq: def U}
U(y):=X(y)- P_l X(y):= A_l y -P_l A_l y.
\end{equation}
then $X(y):=A_l y$ is still a mean 0, variance 1 sub-gaussian random vector. According to the Hoeffding inequality, the sub-gaussian moment of entries of $X(y)$ is bounded above by $CK$ (see Theorem 3.3 in \cite{LecnotesRMMR}). Without ambiguity, we use the notation $U,X$ instead of $U(y),X(y)$.

In step 1, we show with high probability that
$$
\V U\V_2 \lesssim \sqrt{l}, \ \V A^{-1} U\V_2 \gtrsim \frac{\sqrt{n}}{\sqrt{l}}.
$$
First, we give an upper bound for $\V P_l^{\perp}A_l\V$. This leads to an uniform upper bound of $\V U(y) \V_2$ for all $y \in S^{l-1}$.
\newline

{\bf Step 1.1. Concentration of $\V P_l^{\perp} A_l \V$.}

First, notice that $I-P_l=P_l^{\perp}$, which is an orthogonal projection onto $H_l^{\perp}$, so it does not depend on $A_l$ only on $A_{n-l}$. Thus, $P_l^{\perp}$ can be treated as a fixed matrix. We apply Theorem \ref{subgOPconc} with $B=P_l^{\perp}$ and $G=A_l$, then we have
$$
\P (\V P_l^{\perp} A_l \V > \alpha\sqrt{l} )\le 2\exp(-C\alpha^2l).
$$
In particular, for a single vector we have
$$
\P (\V U \V_2 > \alpha\sqrt{l} )= \P (\V P_l^{\perp} A_l y \V_2 > \alpha\sqrt{l} ) \le 2\exp(-C\alpha^2l).
$$
\newline

{\bf Step 1.2. Concentration of $\V A^{-1} U \V_2$.}

Now consider
$$
A^{-1}U=A^{-1} A_l y - A^{-1}P_l A_l y=y-A^{-1}P_l A_l y.
$$
Notice that $A^{-1}P_l A_l y$ is supported in span $\{ e_{l+1}, \cdots, e_n \}$ because $P_l A_l y\in H_l$. So we have
\begin{equation}
\begin{array}{rl}
\V A^{-1} U\V_2^2 & = \V y\V_2^2 + \V A^{-1}P_l A_l y \V_2^2 >\V A^{-1}P_l A_l y \V_2^2\\
& = \displaystyle \sum_{k=1}^{n} \langle A^{-1}P_l A_l y, e_k \rangle^2 = \sum_{k=1}^{n} \langle P_l X, (A^{-1})^T e_k \rangle^2\\
&  =\displaystyle \sum_{k=1}^{n} \langle  P_l X, X_k^* \rangle^2 = \sum_{k=l+1}^{n} \langle  P_l X, X_k^* \rangle^2\\
& = \displaystyle \sum_{k=l+1}^{n} \langle  X, P_l X_k^* \rangle^2 = \sum_{k=l+1}^{n} \langle  X, Y_k^* \rangle^2
\end{array}
\end{equation}
where in the third line we used the fact that $(X_k,X_k^*)_{k=1}^n$ forms a complete biorthogonal system on $\mathbb{R}^n$ from Lemma \ref{BOSp}. Thus, $X_k^* \perp H_l, k\le l$.

Using the above property, let $B$ be the $(n-l)\times n$ matrix whose rows are $(Y_k^*)^T, k=l+1,l+2,\cdots,n$. Then we have
$$
\V A^{-1} U\V_2^2 \ge \sum_{k=l+1}^{n} \langle  X, Y_k^* \rangle^2 = \V B X \V_2^2.
$$
Our goal is to get a small ball probability estimate of $\V B X \V$. As $B$ is independent of $X$, we would like to apply Theorem \ref{SBofLIofHDdistr}. Thus, we first need an estimate for $\V B \V$  and $\V B \V_{\HS}$.
\newline

{\bf Step 1.2.1. Lower bound of $\V B \V_{\HS} $.}

According to Theorems \ref{LIofnewBOS}, \ref{newBOS} and \ref{BOSp}, we have
$$
\V B \V_{\HS}^2 = \sum_{k=l+1}^{n} \V Y_k^* \V_2^2 = \sum_{k=l+1}^{n}{\rm dist}(X_k,H_{l,k})^{-2}= \sum_{k=l+1}^{n}\V P_{l,k}^{\perp}X_k \V_2^{-2}.
$$
Denote $V_j={\rm dist}^2(X_j,H_{l,j})$. Then
\begin{equation}
\begin{array}{rl}
\P\left(\V B \V_{\HS} < \alpha^{-1}\sqrt{\displaystyle\frac{n-l}{l}}\right)
& = \P\left( \left( \displaystyle \frac{1}{n-l} \sum_{j=l+1}^n V_j^{-1} \right)^{-1} > \alpha^2 l\right)\\
& \le \P\left(  \displaystyle \frac{1}{n-l} \sum_{j=l+1}^n V_j > \alpha^2 l\right)\\
& = \P\left(  \displaystyle \frac{1}{n-l} \sum_{j=l+1}^n (V_j-4(l+1))) > \alpha^2 l -4(l+1) \right)\\
& \le \P\left(  \displaystyle \frac{1}{n-l} \sum_{j=l+1}^n (V_j-4(l+1)))_+ > \frac{\alpha^2}{2} l\right).
\end{array}
\end{equation}
where the first inequality follows from the inequality between harmonic mean and arithmetic mean and the second inequality is trivial if we provide $\alpha^2 > 10$. Consider
\begin{equation}
\begin{array}{rl}
\P\left( (V_j-4(l+1)))_+ > 4t \right) & \le \P\left( \sqrt{V_j} > \sqrt{l+1}+\sqrt{t} \right)\\
& =\P\left( \V P_{l,j}^{\perp}X_j \V_2 - \sqrt{l+1} > \sqrt{t} \right)
\end{array}
\end{equation}
Then applying Theorem \ref{subgVconc} with $A=P_{l,k}^{\perp}$ we have $M=\sqrt{l+1}$. Thus $(V_j-4(l+1)))_+$ is a sub-exponential random variable with $\V (V_j-4(l+1)))_+ \V_{\psi_1} \le C$. By the triangle inequality,
$$
\left\Vert \displaystyle \frac{1}{n-l} \sum_{j=l+1}^n (V_j-4(l+1)))_+ \right\Vert_{\psi_1}\le C.
$$
Recalling that $l\le \frac{n}{2}$, we have
$$
\P\left(\V B \V_{\HS} < \alpha^{-1}\sqrt{\frac{n}{l}}\right)\le \exp( - C\alpha^2 l).
$$
\newline

{\bf Step 1.2.2. Upper bound of $\V B \V$.}

First, we have
\begin{equation}
\begin{array}{rl}
\V B \V^2 & = \displaystyle\sup_{x\in S^{n-1}-\{0\}} \V Bx \V_2^2= \sup_{x\in \R^n-\{0\}} \frac{\V Bx \V_2^2}{\V x \V_2^2} = \sup_{\V Bx \V_2 =1} \frac{1}{\V x \V_2^2} \\
&=  \displaystyle\sup_{\substack{x\in \R^n-\{0\} \\ \sum_{k=l+1}^{n} \langle  x, Y_k^* \rangle^2=1}} \frac{1}{\V x \V_2^2}
=\displaystyle\sup_{\substack{x\in H_l-\{0\} \\ \sum_{k=l+1}^{n} \langle  x, Y_k^* \rangle^2=1}} \frac{1}{\V x \V_2^2}
\\
\end{array}
\end{equation}
where the last equality can be justified by considering the decomposition  $x=x'+x'',x'\in H_l,x''\in H_l^{\perp}$ with $\V Bx\V_2=1$. Since $Bx',Bx$ have the same $L_2$ norm and $x'$ has a smaller $L_2$ norm, the supremum must be achieved on $H_l$. Consider $x=A_{n-l}z, z\in S^{n-l-1}$, then
$$
\langle  x, Y_k^* \rangle^2=\langle  A_{n-l}z, P_l X_k^* \rangle^2= \langle  P_l A_{n-l}z,  X_k^* \rangle^2 = \left\langle   \displaystyle\sum_{k=l+1}^{n}z_k X_k,  X_k^* \right\rangle^2 = z_k^2.
$$
Thus, we have
$$
\V B \V^2= \displaystyle\sup_{\substack{z\in S^{n-l-1} }} \frac{1}{\V A_{n-l}z \V_2^2}= s_{n+1-l}(A_{n-l})^{-2}.
$$
By Theorem \ref{leastSVLB}, we have
$$
\P \left( \V B \V > \alpha \frac{\sqrt{n}}{l} \right)
=\P \left( s_{n+1-l}(A_{n-l}) < \alpha^{-1} \frac{l}{\sqrt{n}} \right) \le (C\alpha^{-1})^l +\exp(-Cn).
$$
\newline

{\bf Step 1.2.3. Concentration of $\V B X\V$.}

By Lemma \ref{LIofnewBOS}, $B$ is independent to $X$. So we may condition on $B$ such that $\V B \V_{\HS} > \alpha^{-1}\sqrt{\frac{n}{l}}$ and $\V B \V < \alpha \frac{\sqrt{n}}{l}$.  By equation (\ref{concCondition}) and Theorem \ref{ConofSum},
\begin{equation}
  \L(X_i, \beta ) \le c p \beta,\  {\rm for} \ {\rm  all}\  i \in [n].
\end{equation}
So, we may apply Theorem \ref{SBofLIofHDdistr} with $\vv=\frac{1}{2}$ and have
\begin{equation}
\begin{array}{rl}
\P\left( \displaystyle \V BX\V_2 \le \beta\alpha^{-1} \sqrt{\frac{n}{l}}\right) & \le \P\left( \V BX\V_2 \le \beta \V B  \V_{\HS}\right) \\
& \displaystyle \le (C\beta)^{cr(B)}\le \left( C\beta \right)^{\frac{l}{2\alpha^{4}}}.
\end{array}
\end{equation}
\newline

{\bf Conclusion of step 1.}
Consider the events,
\begin{equation}
\begin{array}{rl}
\mathcal{E}_1 & := \left\{A: \V P_l^{\perp} A_l \V > \alpha\sqrt{l} \right\}\\
\mathcal{E}_2 & := \left\{A: \V B \V_{\HS} < \alpha^{-1}\sqrt{\frac{n}{l}},\ {\rm and}\ \V B \V > \alpha \frac{\sqrt{n}}{l} \right\}.
\end{array}
\end{equation}
We have shown
\begin{equation}
\begin{array}{rl}
\P(\mathcal{E}_1) & \le 2\exp(-C\alpha^2l)\\
\P(\mathcal{E}_2) & \le \exp( - C\alpha^2l)+(C\alpha^{-1})^l +\exp(-Cn)
\end{array}
\end{equation}
By conditioning on $\mathcal{E}_2^c$ for all $ y \in S^{l-1}$ and a vector $U$ defined in \eqref{eq: def U}, we have
$$
\displaystyle \P \left(  \V A^{-1} U\V_2 < \beta\alpha^{-1} \sqrt{\frac{n}{l}} \bigg\vert \mathcal{E}_2^c \right)
< \left( C\beta \right)^{\frac{l}{2\alpha^{4}}}.
$$
\newline

{\bf Step 2: Preparation for the union bound argument.}

Now let $E_1$(or in fact $\R^{l'}$) be an $l':=\lfloor \kappa l \rfloor$ dimensional coordinate subspace that is spanned by $e_1,\cdots,e_{l'}$. We consider an $\vv-$net $\N_{\vv}$ on $S^{l'-1}$(i.e. $S_{E_1}$), then $\v \N_{\vv}\v\le (3\vv^{-1})^{l'}$. And for all $y_i \in \N_{\vv}$, define
$$
U_i=U(y_i):=X(y_i)- P_l X(y_i):= A_l y_i -P_l A_l y_i.
$$

{\bf Step 2.1. $(A_l -P_l A_l)\N_{\vv}$ is a net on some ellipsoid.}

Let
\[
  E_2:=(A_l -P_l A_l)\R^{l'},  \quad S_2:=(A_l -P_l A_l)S^{l'-1}.
\]
By step 1.1, with probability $1-\exp(-C\alpha^2l)$, $\V P_l^{\perp} A_l \V \le \alpha\sqrt{l}$, i.e., $S_2\subset \alpha \sqrt{l} B_2^{l'}$.

Also, consider any cap on $S_2$ of radius $\delta$. Then if $P_l^{\perp} A_l\N_{\vv}$ is not a $\delta\sqrt{l}$-net on $S_2$, there exists some $\delta\sqrt{l}$ cap that does not intersect $P_l^{\perp} A_l\N_{\vv}$. This means that the pre-image of the cap does not intersect $\N_{\vv}$.
However, if $\V P_l^{\perp} A_l \V \le \alpha\sqrt{l}$, then the pre-image contains a cap of radius at least $\frac{\delta}{\alpha}$. Thus, for $\vv = \frac{\delta}{ \alpha}$, with probability $1-\exp(-C\alpha^2l)$, $P_l^{\perp} A_l\N_{\vv}$ is a $\delta\sqrt{l}$-net on $S_2\subset \alpha \sqrt{l} B_2^{E_2}$. We denote this $\delta\sqrt{l}$-net by $\N_{\delta}:=P_l^{\perp} A_l\N_{\vv}$
\newline

{\bf Step 2.2. Reduction of our objective.}

Now, we want to show that for some small choice of $\kappa$, $\V A^{-1}U \V_2 \gtrsim \sqrt{\frac{n}{l}}$, for all $U\in S_2$ with high probability. If we can prove this, then together with step 2.1., we have $s_{l'} (A^{-1}) \gtrsim \frac{\sqrt{n}}{l} $ with high probability.

On the event that $P_l^{\perp} A_l\N_{\vv}$ forms a $\delta\sqrt{l}$ net on $S_2$, we have for all $U\in S_2$, 
there exists some $U_i\in \N_{\delta}$ such that $\V U-U_i \V_2 \le \delta \sqrt{l}$, and
$$
\V A^{-1} U\V_2 \ge \V A^{-1}U_i\V_2 -\V A^{-1}(U_i-U)\V_2.
$$
For the first term, since we have
$$
\v \mathcal{N}_{\delta}\v = \v \mathcal{N}_{\vv}\v \le (3\vv^{-1})^{l'}= \left(\frac{3\alpha}{\delta}\right)^{l'}
$$
we obtain $\V A^{-1}U_i\V_2 \gtrsim \sqrt{\frac{n}{l}}$, for all $U_i$ with high probability by choosing $\kappa$ small.

To bound $\V A^{-1}(U_i-U)\V_2$ from above for $\V U_i-U\V_2 \lesssim \sqrt{l}$, we only have to prove
$$
\left\Vert A^{-1}\vert_{E_2} \right\Vert \lesssim \frac{\sqrt{n}}{l}
$$
with high probability.
\newline

{\bf Step 2.3. Upper bound for $\left\Vert A^{-1}\vert_{E_2} \right\Vert$.}

Notice that
$$
\left\Vert A^{-1}\vert_{E_2} \right\Vert \le
\left\Vert A^{-1}P_l^{\perp} A_{l'} \right\Vert \cdot
\left\Vert (P_l^{\perp}A_{l'})^{-1}: E_2\rightarrow \mathbb{R}^{l'} \right\Vert
=\frac{\left\Vert A^{-1}P_l^{\perp} A_{l'} \right\Vert}{s_{\rm min}(P_l^{\perp}A_{l'}) }.
$$
We only need to prove for $\kappa$ small enough:

\begin{enumerate}
\item $s_{\rm min}(P_l^{\perp}A_{l'}) \gtrsim  \sqrt{l}$ with high probability.
\item $\left\Vert A^{-1}P_l^{\perp} A_{l'} \right\Vert \lesssim \sqrt{\frac{n}{l}}$ with high probability.
\end{enumerate}

{\bf Step 2.3.1. Lower bound of $s_{\rm min}(P_l^{\perp}A_{l'})$.}

First, by Theorem \ref{subgOPconc},
$$
\P\left( \V P_l^{\perp}A_{l'} \V \ge \alpha\sqrt{l}\right) \le 2\exp(-C\alpha^2l).
$$
Next, consider a $\frac{1}{2\alpha^2}$-net $\N$ on $S^{l'-1}$, then $|\N| \le (6\alpha^2)^{l'}$. And for all $ y_i \in \N$, consider $A_{l'}y_i$ as a random vector.
We use an elementary inequality $\L(Z,mt) \le m \L(Z,t)$ which holds for any $m \in \mathbb{N}$.
Consider $\beta < \alpha^{-1}$, then by equation (\ref{concCondition}) and Theorem \ref{ConofSum},
\begin{equation}
  \L((A_{l'}y_i)_j, \alpha^{-1} ) \le c p \beta \left\lfloor \frac{\alpha^{-1}}{\beta} +1 \right\rfloor \le 2 cp \alpha^{-1} , \ {\rm for }\ {\rm all}\  j \in [n].
\end{equation}

$P_l^{\perp}$ is decided by $A_{n-l}$, which is independent with $A_{l'}y_i$. So we may consider $P_l^{\perp}$ as a fixed matrix and apply Theorem \ref{ConcofProj} to obtain
$$
\displaystyle \P\left( \V P_l^{\perp}A_{l'} y_i \V_2 \le \alpha^{-1}\sqrt{l} \right) \le (C\alpha^{-1})^{l}.
$$
Let $y\in S^{l'-1}$ and choose  $ y_i\in \mathcal{N}$ with $\V y-y' \V_2 < \frac{1}{2\alpha^2}$. Conditioning on $A$ such that $\V P_l^{\perp}A_{l'} \V \le \alpha\sqrt{l}$, then we have
\begin{equation}
\begin{array}{rl}
\V P_l^{\perp}A_{l'}y \V_2 & \ge \V P_l^{\perp}A_{l'}y_i \V_2 - \V P_l^{\perp}A_{l'}\V \V y-y_i \V_2\\
 & \ge \alpha^{-1} \sqrt{l} - \frac{1}{2\alpha^2}\alpha\sqrt{l}=\frac{1}{2\alpha}\sqrt{l}\\
\end{array}
\end{equation}
with probability $1-(C\alpha^{-1})^{l}$.

Thus, with a standard union bound argument, we have
$$
\P\left( s_{\rm min}\left( P_l^{\perp}A_{l'} \right) \ge  \frac{1}{2\alpha}\sqrt{l} \right)\ge 1-2\exp(-C\alpha^2l)-\left(6\alpha^2\right)^{l'}(C\alpha^{-1})^{l}.
$$
\newline

{\bf Step 2.3.2. Upper bound of $\left\Vert A^{-1}P_l^{\perp} A_{l'} \right\Vert$.}

Recall for $y\in S^{l'-1}$
$$
\V A^{-1} P_l^{\perp}A_{l'} y\V_2^2  = \V y\V_2^2 + \V A^{-1}P_l A_l y \V_2^2 = 1+\V BA_{l'} y\V_2^2
$$
where
$$
\V B \V_{\HS}^2=\sum_{k=l+1}^{n}{\rm dist}(X_k,H_{l,k})^{-2}.
$$
Thus, we only need to show $\V BA_{l'} \V \lesssim \sqrt{\frac{n}{l}}$. We will prove this using Theorem \ref{subgOPconc}. To apply Theorem \ref{subgOPconc}, we employ an argument that is presented in \cite[Section 5.4.1. and Section 13.2.]{Nogapsdeloc}. This argument provides an upper estimate of $\V B \V_{\HS}$.

Recall that the weak $L^p$ norm of a random variable $Z$ is defined as
$$
\V Z \V_{p,\infty} =\sup_{t>0} t\cdot \left(\P \left\{ |Z|>t\right\} \right)^{1/p}.
$$
Although it is not a norm, it is equivalent to a norm if $p>1$. In particular, the weak triangle inequality holds:
$$
\displaystyle \left \Vert \sum_i Z_i \right\Vert_{p,\infty} \le C(p) \sum_i \V Z_i \V_{p,\infty}
$$
where $C(p)$ is bounded above by an absolute constant for $p\ge 2$, see \cite{EMSteinGWeiss}, Theorem 3.21.

Now by Theorem \ref{disttoSubsp}, for any $t>0$,
$$
\P \left\{{\rm dist}(X_k,H_{l,k}) \le t\sqrt{l}\right\}  \le (Ct)^l+\exp(-Cn).
$$
Define
$$
W_k: =\min\left(\text{dist}(X_k, H_{l,k})^{-2}, (t_0 \sqrt{l})^{-2}\right)
$$
where $t_0=\frac{C_0l}{n}$ and $C_0$ is a small constant depending only on $K$.
Then we have
\begin{equation}
\begin{array}{rl}
\displaystyle \left \Vert  W_k \right\Vert_{l/2,\infty} & =
\sup_{t>0} t\cdot \left(\P \left\{ W_k >t\right\} \right)^{2/l}\\
& = \sup_{t>0} t^{-2}l^{-1} \cdot \left(\P \left\{ W_k^{-\frac{1}{2}} <t\sqrt{l}\right\} \right)^{2/l}\\
& = \sup_{t>t_0 } t^{-2}l^{-1} \cdot \left(\P \left\{ {\rm dist}(X_k,H_{l,k}) <t\sqrt{l}\right\} \right)^{2/l}\\
\displaystyle & \le \frac{C}{l}+\frac{1}{t_0^2 l}\exp\left( -C\frac{n}{l} \right)\\
& \le  \frac{C}{l}+\frac{1}{l}\left(  \frac{n^2}{C_0^2 l^2} \exp\left( -C\frac{n}{l} \right) \right) \le \frac{C}{l}.
\end{array}
\end{equation}
this implies
$$
\left\Vert \sum_{k=l+1}^{n}W_k \right\Vert_{l/2,\infty} \le \frac{C(n-l)}{l}\le \frac{Cn}{l}.
$$
Thus, we have
$$
\P \left\{\sum_{k=l+1}^{n}W_k > t^2\frac{n}{l}\right\} \le (Ct^{-1})^l.
$$
On the other hand,
\begin{equation}
\begin{array}{rl}
\displaystyle & \P \left({\rm there}\ {\rm exists}\  k, \ W_k \neq \text{dist}(X_k, H_{l,k})^{-2}\right)\\
 \le & \displaystyle\sum_{k=l+1}^{n}
\P \left\{{\rm dist}(X_k,H_{l,k})  \le t_0\sqrt{l}\right\} \\
\displaystyle \le & (n-l)\left( \left( \frac{CC_0l}{n} \right)^l+\exp(-Cn) \right)\\
\le & \exp(-Cl).
\end{array}
\end{equation}
So we have
\begin{equation}
\begin{array}{rl}
& \displaystyle\P\left\{\V B \V_{\HS} > t \sqrt{\frac{n}{l}}\right\}\\
\le & \displaystyle \P \left\{\sum_{k=l+1}^{n}W_k > t^2\frac{n}{l}\right\}+
\P \left({\rm there}\ {\rm exists} \  k, \ W_k \neq \text{dist}(X_k, H_{l,k})^{-2}\right)
 \\
\le & (Ct^{-1})^l +\exp(-Cl).
\end{array}
\end{equation}

Now, denote
$$
\mathcal{E}_2':=\mathcal{E}_2\cup \left\{ A: \V B \V_{\HS} > \alpha \sqrt{\frac{n}{l}} \right\}.
$$
Applying Theorem \ref{subgOPconc} with $D=B,G=A_{l'},Cs=\frac{1}{2}\alpha,Ct=\frac{1}{2}\alpha\sqrt{\frac{l}{l'}}$,  we have for $\alpha$ large enough,
\begin{equation}
\begin{array}{rl}
& \displaystyle \P\left\{\V BA_{l'} \V  > \alpha^2\sqrt{\frac{n}{l}} \bigg\vert (\mathcal{E}_2')^c
 \right\}\\
\le &  \displaystyle \P\left\{\V BA_{l'} \V  > \frac{1}{2}\alpha\V B \V_{\HS} +
\left( \frac{1}{2} \alpha\sqrt{\frac{l}{l'}} \right) \sqrt{l'}\V B \V \bigg\vert (\mathcal{E}_2')^c
 \right\}\\
\le & 2\exp\left(-C\alpha^2l\left( \alpha^{-4} +1\right) \right)
\le 2\exp\left(-C\alpha^2l \right).
\end{array}
\end{equation}
So we have
$$
\P\left( \V A^{-1} P_l^{\perp} A_{l'} \V \ge \alpha^2 \sqrt{\frac{n}{l}} \bigg\vert (\mathcal{E}_2')^c \right) \le 2\exp\left(-C\alpha^2l \right).
$$
\newline

{\bf Conclusion of Step 2.}

Denote
\begin{equation}
\begin{array}{rl}
\mathcal{E}_3 & :=\left\{A: \left\Vert A^{-1}\vert_{E_2} \right\Vert
\ge     2\alpha^3 \frac{\sqrt{n}}{l}
\right\}\cup \mathcal{E}_1 \cup \mathcal{E}_2',\\
 \mathcal{E}_4 &:=\left\{A: \displaystyle {\rm there}\ {\rm exists} \  y_i \in \mathcal{N}_{\vv} {\rm \ such \ that \ }
 \V A^{-1}U(y_i) \V_2 \le \beta\alpha^{-1}\sqrt{\frac{n}{l}}  \right\}.
 \\
\end{array}
\end{equation}
Then we  have
$$
\P\left( \mathcal{E}_3 \right)\le
\left(6\alpha^2\right)^{l'}(C\alpha^{-1})^{l} +
4\exp(-C\alpha^2l)+\exp(-Cl)
+\P\left( \mathcal{E}_2' \right)
+\P\left( \mathcal{E}_1 \right).
$$
Since $\v N_{\vv }\v \le \left( \frac{\alpha\sqrt{l}}{\delta}\right)^{l'-1}$ as we discussed in step 2.1,
$$
\P\left( \mathcal{E}_4  \mid  \mathcal{E}_2^c\right) \le
\left(\frac{3\alpha\sqrt{l}}{\delta}\right)^{l'} \left( C\beta \right)^{\frac{ l}{2\alpha^4}}
$$
\newline

{\bf Step 3. The union bound argument.}

Denote
\begin{equation}
\begin{array}{rl}
\mathcal{E}:= &\bigg\{ A:
{\rm there} \ {\rm exists} \ y\in S^{l'-1},{\rm \ such \ that \ } \V A^{-1} U(y)\V_2 \le \frac{\beta}{2\alpha} \sqrt{\frac{n}{l}}, \\
& {\rm \ or\ } \V U(y) \V_2 \ge \alpha \sqrt{l}
\bigg\}.
\end{array}
\end{equation}
Choose $\delta$  $2\alpha^4\delta=\frac{\beta}{2}$.
 Let $y\in S^{l'-1}$ and choose  $ y_i\in \mathcal{N}_{\vv}$ with $\V y-y' \V_2 < \delta$. If $ A \notin \mathcal{E}_3\cup\mathcal{E}_4$, then $\V U(y) \V \le \alpha \sqrt{l}$ and
\begin{equation}
\begin{array}{rl}
\V A^{-1} U(y)\V_2 & \ge \V A^{-1}U(y)\V_2 -\V A^{-1}(U(y_i)-U(y))\V_2\\
& \ge \beta\alpha^{-1}\sqrt{\frac{n}{l}}- 2\alpha^3 \frac{\sqrt{n}}{l}\cdot \delta\\
& \ge \frac{\beta}{2\alpha}\sqrt{\frac{n}{l}}.
\end{array}
\end{equation}
Thus, we have $\mathcal{E}\subset \mathcal{E}_3\cup \mathcal{E}_4$. On the other hand,
\begin{equation}
\begin{array}{rl}
A\in \mathcal{E}^c
& \Rightarrow s_{l'}(A^{-1}) \ge \frac{\beta}{2\alpha^2}\frac{\sqrt{n}}{l}
\Rightarrow s_{n+1-l'}(A) \le \frac{2\beta}{\alpha^2}\frac{l}{\sqrt{n}}\\
& \Rightarrow s_{n+1-l'}(A) \le \frac{4\beta}{\kappa\alpha^2}\frac{l'}{\sqrt{n}}
, \ {\rm for }\ {\rm all}\  l'< \frac{\kappa \tilde{c} n}{2}.
\end{array}
\end{equation}
The $\frac{\tilde{c}}{2}$ in $l'< \frac{\tilde{c}\kappa n}{2}$ comes from the requirement that $l\le \frac{\tilde{c}n}{2}$.

Now let $\beta=\exp(-\alpha^5)$, then $\delta=\frac{1}{4} \alpha^{-4} \exp(-\alpha^5)$. Choose $\alpha $ to be a big enough constant, then
\begin{equation}
\begin{array}{rl}
\displaystyle & \P\left( \mathcal{E} \right)\\
\le &   \P\left( \mathcal{E}_3 \right)+ \P \left( \mathcal{E}_4 \right)\\
\displaystyle\le &
\left(\frac{3\alpha}{\delta}\right)^{l'} \left( C\beta \right)^{\frac{ l}{2\alpha^4}}
+7 \exp(-C\alpha^2l) + 2(C\alpha^{-1})^l+\exp(-Cl)  +\exp(-Cn)\\
\displaystyle& + \left(6\alpha^2\right)^{l'}(C\alpha^{-1})^{l}\\
\displaystyle \le &
\left(\left( 12 \alpha^5 \exp(\alpha^5) \right)^{\kappa}
C \exp(-\frac{1}{2}\alpha)
\right)^l
+
\left( C \alpha^{2\kappa - 1}
\right)^l + 9 (C\alpha^{-1})^l  +\exp(-Cl).\\
\end{array}
\end{equation}
Replace $l',l$ by $l,\kappa^{-1}l$, then for a sufficiently small $\kappa$ depending on $\alpha$, and $l < \frac{\kappa \tilde{c} n}{2}$, we have
\begin{equation}
\begin{array}{rl}
& \P\left( s_{n+1-l}(A) \ge \frac{4\beta}{\kappa\alpha^2}\frac{l}{\sqrt{n}} \right)\\
\le & \left( \exp(-\frac{1}{4}\alpha)
\right)^{\kappa^{-1}l}
+ \left( C \alpha^{ - \frac{1}{2}}
\right)^{\kappa^{-1}l} + 9 (C\alpha^{-1})^{\kappa^{-1}l}  +\exp(-Cl)\\
\le & (C\alpha^{-\frac{1}{2}})^{\kappa^{-1}l}  +\exp(-Cl).
\end{array}
\end{equation}
Choosing a sufficiently large $\alpha$, we show that there exist $ C_1, C_2, C_3 $  depending only on $K,p$, such that
$$
\P\left( s_{n+1-l}(A) \ge C_1\frac{l}{\sqrt{n}} \right)
\le \exp(-C_2 l )
$$
for all $l \le C_3 n.$ For $l>C_3n$, the above bound follows from the estimate for $s_1(A)$. So we have for all $1\le l\le n$,
$$
\P\left( s_{n+1-l}(A) \ge C_1\frac{l}{\sqrt{n}} \right)
\le \exp(-C_2 l ).
$$
\end{proof}
\begin{remark}{\rm
As the proof demonstrates, Assumption \ref{asump1} is satisfied with $s=\beta$ which only depends on $p$ and $K$. }
\end{remark}

\begin{remark}{\rm The only place we used the non-degeneracy condition is in the application of \ref{SBofLIofHDdistr}. We expect that the same result holds without the concentration function condition. To remove that condition, the application of Theorem \ref{SBofLIofHDdistr} on $BX$ must be replaced by showing matrix $B$ does not have a good arithmetic structure with high probability (for arithmetic structure of random matrices and its application, see \cite{InvofRM, LecnotesRMMR, LOandInvofRM, NAtheoryofRM, Nogapsdeloc} ).
}\end{remark}

\section{Deduction of Corollary \ref{upperboundRect} and \ref{mainthm}}
Both Theorem \ref{upperboundRect} and \ref{mainthm} are direct corollaries of Theorem \ref{upperbound}.
\begin{proof}[Proof of Theorem \ref{upperboundRect}]
Construct an $n\times n$ random matrix $J$ such that its first $n-k$ rows are matrix $A$ and the rest are i.i.d. entries with the same distribution as $A_{i,j}$. Then, by Theorem \ref{upperbound}, for all $ t>0$ and $k$ between $l$ and $n$, $s_{n+1-l}(J)\le \frac{C_1 t l}{\sqrt{n}}$, with probability $1-\exp(-C_2 t l)$, where $C_1,C_2$ are constants that depend only on $K,p$. This implies, with the same probability, there exists an $l$-dimensional subspace $E$ such that $\V Jy \V_2 \le \frac{C_1tl}{\sqrt{n}}$ for all $ y\in S_E$.

Let $F:={\rm span}\left\{ e_1, \cdots, e_{n-k} \right\}$, then $\V Jy \V_2 \le \frac{C_1 tl}{\sqrt{n}}$, for all $y\in S_{E\cap F}$ with  probability $1-\exp(-C_2 t l)$. This implies
$$
\P_J \left\{ s_{n+1-l}(A) \ge \frac{C_1 tl}{\sqrt{n}} \right\} \le \exp(-C_2 t l)
$$
Then we only need to notice that the above event is independent of the last $k$ rows of $J$; thus the probability is also with respect to $A$.
\end{proof}

To prove Theorem \ref{mainthm},  in addition to applying Theorem \ref{upperbound}, we only need Theorem \ref{leastSVLB} to give a lower bound.
\begin{proof}[Proof of Theorem \ref{mainthm}]
For the lower bound, denote $J$ as the matrix of the first $n-l$ rows of $A$. Then we have by Theorem \ref{leastSVLB}
\begin{equation}
\begin{array}{rl}
\P\left\{ s_{n+1-l}(A) < \frac{C_1l}{\sqrt{n}}  \right\}
&\le \P\left\{ s_{n+1-l}(J) < \frac{C_1l}{\sqrt{n}}  \right\}\\
&\le (CC_1)^l+\exp(-Cn) \le  \frac{1}{2}\exp(-\frac{1}{2}C_3l)+\exp(-Cn)\\
&\le \frac{1}{2}\exp(-C_3l)
\end{array}
\end{equation}
with some small constant $C_3$.
The upper bound follows directly with a large $t$ in Theorem \ref{upperbound}.

\end{proof}
Note that Theorem \ref{mainthm} is a generalization of Theorem 1.3 in \cite{DisttoLinfSZ}. Theorem \ref{mainthmrect} can be proved in the same way as Theorem \ref{mainthm}.

\section{Acknowledgement}
I want to thank my advisor Professor Mark Rudelson for many helpful discussions and advice.

\begin{appendices}

\section{Proof of Theorem \ref{BOSp}}

\newtheorem*{thminap}{Theorem  \ref{BOSp}}
\begin{thminap}
\begin{enumerate}
\item  Let $D$ be an $n\times n$ invertible matrix with columns $v_k=De_k$, $k=1,2,\cdots,n$. Define $v_k^*=(D^{-1})^*e_k$. Then $(v_k,v_k^*)_{k=1}^n$ is a complete biorthogonal system in $\mathbb{R}^n$.\\
\item  Let $(v_k)_{k=1}^n$ be a linearly independent system in an $n-$dimensional Hilbert space $H$. Then there exist unique vectors $(v_k^*)_{k=1}^n$ such that $(v_k,v_k^*)_{k=1}^n$ is a biorthogonal system in $H$. This system is complete.\\
\item  Let $(v_k,v_k^*)_{k=1}^n$ be a complete biorthogonal system in a Hilbert space $H$. Then $\V v_k^* \V_2 = 1/{\rm dist}(v_k,{\rm span}(v_j)_{j\neq k}) $ for $k=1,2,\cdots,n$.
\end{enumerate}
\end{thminap}
\begin{proof}
(1) follows directly  from
$$
\langle v_j,v_k^* \rangle= \langle De_j,(D^{-1})^*e_k \rangle
=\langle D^{-1}De_j,e_k \rangle=\langle e_j,e_k \rangle=\delta_{j,k}.
$$
To prove (2), we use the fact that any basis on a finite dimensional vector space has a unique dual basis. Since $H$ is Hilbert space, the dual basis also belongs to $H$ which forms a biorthogonal system together with the original basis. The completeness follows from the dimension argument.

Since $(v_i,v_i^*)_{i=1}^n$ is a complete biorthogonal system on the Hilbert space $H$, for any $k=1,2,\cdots,n$, we have $v_k^*\perp \span\{v_i:i\neq k\}$ and $\{v_i: i=1,\cdots,n, i \neq k\}\cup \{{v_k^*}\}$ form a basis on $H$. Thus we have the decomposition
$$
v_k=\sum_{i\neq k} a_i v_i + {\rm dist}(v_k,{\rm span}(v_j)_{j\neq k})\frac{v_k^*}{\V v_k^*\V_2}.
$$
Take inner product with $v_k^*$ at both sides and we have
$$
1= {\rm dist}(v_k,{\rm span}(v_j)_{j\neq k})\V v_k^*\V_2,
$$
which proves (3).

\end{proof}

\end{appendices}


\renewcommand\refname{Reference}
\bibliographystyle{plain}
\bibliography{feng}

\end{document}